\documentclass[reqno,final,authoryear]{amsart}

\usepackage[utf8]{inputenc}

\usepackage[UKenglish]{babel}
\usepackage{amssymb,amsmath,amsfonts,amsthm}

\usepackage{a4wide}
\usepackage{graphicx}

\usepackage{psfrag}

\usepackage{natbib}
\usepackage{txfonts}

\usepackage{xcolor}
\usepackage{hyperref}
\hypersetup{
colorlinks,
citecolor=violet,
linkcolor=red,
urlcolor=blue
}
\usepackage[capitalise]{cleveref}

\newtheorem{theorem}{Theorem}[section]
\newtheorem{lemma}[theorem]{Lemma}

\newcommand{\dd}{\mathrm{d}}
\newcommand{\ee}{\mathrm{e}}

\renewcommand{\leq}{\leqslant}
\renewcommand{\geq}{\geqslant}

\begin{document}

\title{On the expected number of successes in a sequence of nested\\ Bernoulli trials}
\author{Eckhard Schlemm}
\address{Wolfson College, University of Cambridge, United Kingdom}
\email{es555@cam.ac.uk}
\begin{abstract}
We analyse the asymptotic behaviour of the probability of observing the expected number of successes at each stage of a sequence of nested Bernoulli trials. Our motivation is the attempt to give a genuinely frequentist interpretation to the notion of probability based on finite sample sizes. The main result is that the probabilities under consideration decay asymptotically as $n^{-1/3}$, where $n$ is the common length of the Bernoulli trials. The main ingredient in the proof is a new fixed-point theorem for non-contractive symmetric functions of the unit interval.
\end{abstract}

\maketitle

\section{Introduction and main results}

In a frequentist interpretation, the probability of an event is defined as its asymptotic relative frequency in a large number of independent experiments. In modern axiomatic probability theory, this interpretation is reflected in various forms of the law of large numbers. We refer the reader to any standard text book of probability theory for a technical discussion of these topics and to \citet{vonmises1981probability} for a more philosophical account. Frequentism suggests that a sequence of $n_1$ independent experiments with individual probability of success $p$ (such as the tossing of a biased coin), would yield, on average, $n_1p$ successes. This is reflected by the fact that the number of successes in this setup follows a $\operatorname{Bin}(n_1,p)$ binomial distribution which assigns probability $p_{m_1;n_1}^{(1)}=\binom{n_1}{m_1}p^{m_1}(1-p)^{n_1-m_1}$ to the event of observing $m_1$ successes, and has expected value $n_1p$. In a genuinely frequentist approach, these probabilities should be 
interpreted, again, as limits of relative frequencies. More precisely, if the sequence of $n_1$ independent experiments were to be repeated, 
independently, $n_2$ times, then,
on average, one would observe $n_2 p_{m_1,n_1}^{(1)}$ runs with $m_1$ successes. In fact, for each $m_1=1,\ldots,n_1$, the number of runs with exactly $m_1$ successes follows a $\operatorname{Bin}(n_2,p_{m_1,n_1}^{(1)})$ binomial distribution, which is defined by the probabilities
\begin{equation*}
p_{m_1,m_2;n_1,n_2}^{(2)}=\binom{n_2}{m_2}\left(p_{m_1,n_1}^{(1)}\right)^{m_2}\left(1-p_{m_1,n_1}^{(1)}\right)^{n_2-m_2}
\end{equation*}
of observing $m_2$ runs with $m_1$ successes. Iteration of this process leads to the recursive definition
\begin{equation*}
p_{m_1,\ldots,m_k;n_1,\ldots,n_k}^{(k)}=\binom{n_k}{m_k}\left(p_{m_1,\ldots,m_{k-1};n_1,\ldots,n_{k-1}}^{(k-1)}\right)^{m_k}\left(1-p_{m_1,\ldots,m_{k-1};n_1,\ldots,n_{k-1}}^{(k-1)}\right)^{n_k-m_k}.
\end{equation*}

In the following we restrict our attention to the special case where the numbers $n_k$ are all equal to some $n$ and the numbers $m_k$ are equal to the expected number of successes at stage $k$, i.\,e.\ $m_1=np$, $m_2=np_{np;n}^{(1)}$, and so on. The numbers $m_k$ will not, in general, be integers, unless $p$ is rational and $n$ is sufficiently large. This could be remedied by considering the integer closest to $m_k$ instead, but we will not do that here. The subject of the paper is an asymptotic analysis of the array of numbers $p_{k,n}$ defined recursively by $p_{0,n}=p$ and
\begin{equation}
\label{eq-definition-pkn}
p_{k,n} = \binom{n}{np_{k-1,n}}\left(p_{k-1,n}\right)^{np_{k-1,n}}\left(1-p_{k-1,n}\right)^{n(1-p_{k-1,n})},\quad k\geq 1.
\end{equation}
Surprisingly, these numbers do not seem to have received any attention in the past. They arise very naturally, however. They are the probabilities that, in a nested series of Bernoulli experiments, the number of successes coincides at each stage with the expected number of successes. Classically, the binomial coefficient $\binom{n}{m}$ is defined for positive integers $m\leq n$ by the formula $n!/m!(n-m)!$. We extend this definition to the case of real numbers by replacing the factorials in the denominator by Gamma functions, i.\,e.\ $\binom{n}{\alpha}=n!/\Gamma(\alpha+1)\Gamma(n-\alpha+1)$, $0<\alpha\leq n$. For two sequences $a_n$, $b_n$ of positive real numbers we write $a_n\sim b_n$ if $\lim_{n\to\infty}a_n/b_n=1$.

With the following result we initiate the study of the probabilities $p_{k,n}$.

\begin{theorem}
\label{thm-k-fixed}
For every positive integer $k$ and every $p\in(0,1)$, there exist $\alpha_k$ and $\beta_k$ such that $p_{k,n}\sim \alpha_k(2\pi n)^{-\beta_k}$ as $n\to\infty$. The numbers $\alpha_k$ are given by $\alpha_k=[p(1-p)]^{(-1/2)^k}$. The rates $\beta_k$ do not depend on the initial value $p_{0,n}=p$ and are given by $\beta_k=[1-(-1/2)^k]/3$; in particular, $\alpha_k$ and $\beta_k$ converge to $1$ and $1/3$, respectively.
\end{theorem}
Interestingly, the rates $\beta_k$ are related to the well-known Jacobsthal numbers $J_k$ (\href{http://oeis.org/A001045}{OEIS A001045}) via $2^k\beta_k = J_k$. The trivial cases $p\in\{0,1\}$ are easily dealt with separately and are seen to lead to $p_{k,n}\equiv 1$. In the next result we look at the array $(p_{k,n})$ from a different angle and consider the case where $k$ tends to infinity while $n$ is held constant.
\begin{theorem}
\label{thm-n-fixed}
For every positive integer $n$ and every $p\in(0,1)$, the probabilities $p_{k,n}$ converge, as $k\to\infty$, to a limit $p_n\in(0,1)$. This limit is independent of $p$ and is characterised by being the unique solution of the fixed-point equation
\begin{equation}
\label{eq-fixed-point}
p_n=\binom{n}{np_n}p_n^{np_n}(1-p_n)^{n(1-p_n)}.
\end{equation}
Furthermore, $p_n\sim (2\pi n)^{-1/3}$, as $n\to\infty$.
\end{theorem}

\section{Proofs}

In this section we present proofs of \cref{thm-k-fixed} and \cref{thm-n-fixed}. We will repeatedly use Stirling's approximation for factorials (or the Gamma function).
\begin{lemma}[Stirling's approximation]
\label{lem-Stirling-approx}
For every positive integer $n$, the factorial $n!$ satisfies
\begin{equation}
\label{eq-Stirling-approx}
\sqrt{2\pi}n^{n+1/2}\ee^{-n} \leq n! \leq \ee n^{n+1/2}\ee^{-n},\quad \text{and}\quad n!\sim \sqrt{2\pi}n^{n+1/2}\ee^{-n}.
\end{equation}
An analogous approximation holds for the Gamma function $\Gamma$. In particular, for every positive integer $n$ and every positive rational number $\alpha$ not exceeding $n$, the binomial coefficient $\binom{n}{\alpha}$ satisfies
\begin{equation}
\frac{\sqrt{2\pi}n^{n+1/2}\ee^{-n}}{\ee^2\alpha^{\alpha+1/2}\ee^{-\alpha}(n-\alpha)^{n-\alpha+1/2}\ee^{-n-\alpha}} \leq \binom{n}{\alpha} \leq \frac{\ee n^{n+1/2}\ee^{-n}}{2\pi \alpha^{\alpha+1/2}\ee^{-\alpha}(n-\alpha)^{n-\alpha+1/2}\ee^{-n-\alpha}}
\end{equation}

\end{lemma}
\begin{proof}
Immediate consequences of \citet{robbins1955remark}.
\end{proof}

We now give the proof of our first main result.

\begin{proof}[Proof of \cref{thm-k-fixed}]
Induction. We first consider the base case. Here, Stirling's approximation (\cref{lem-Stirling-approx}) applied to \cref{eq-definition-pkn} with $k=1$ shows that
\begin{equation*}
p_{1,n} \sim \frac{1}{\sqrt{2\pi p(1-p)}\sqrt{n}},
\end{equation*}
and thus $p_{1,n}\sim \alpha_1(2\pi n)^{-\beta_1}$ with $\alpha_1=1/\sqrt{p(1-p)}$ and $\beta_1=1/2$. Now, assuming that $p_{k-1,n}\sim \alpha_{k-1}(2\pi n)^{-\beta_{k-1}}$, we prove the corresponding statement for $p_{k,n}$, $k>1$. The same approximation as before, applied to \cref{eq-definition-pkn}, yields
\begin{equation*}
p_{k,n}\sim \frac{1}{\sqrt{2\pi p_{k-1,n}(1-p_{k-1,n})}\sqrt{n}}\sim \frac{1}{\sqrt{\alpha_{k-1}}}(2\pi n)^{-(1-b_{k-1})/2},
\end{equation*}
which proves the first part of the theorem. It also shows that the numbers $\alpha_k$ and $\beta_k$ satisfy the recursions
\begin{equation*}
\alpha_k=\frac{1}{\sqrt{\alpha_{k-1}}},\quad \beta_k=(1-\beta_{k-1})/2,
\end{equation*}
which are easily solved by the reader's favourite method.
\end{proof}

In the next lemma, we will investigate the function $P_n:p\mapsto \binom{n}{np}p^{np}(1-p)^{n(1-p)}$, which governs the recursion \labelcref{eq-definition-pkn}.

\begin{lemma}
\label{lem-convexity}
For every positive integer $n$, the function $P_n$ is convex.
\end{lemma}
\begin{proof}
We will prove the stronger claim that $P_n$ is log-convex. An easy calculation yields
\begin{equation*}
\frac{1}{n}\frac{\dd^2}{\dd p^2}\log P_n(p) = \frac{1}{p(1-p)}-n\left[\psi^{(1)}(1+np)+\psi^{(1)}(1+n(1-p))\right],
\end{equation*}
where $\psi^{(\nu)}$ denotes the polygamma function of order $\nu$. See, for instance, \citet[Section 6.4.]{abramowitz1992handbook} for an introduction to, and basic properties of, these functions. To prove log-convexity of $P_n$, we thus need to argue that for all $p\in[0,1]$, $1/[p(1-p)]$ is greater than $n\left[\psi^{(1)}(1+np)+\psi^{(1)}(1+n(1-p))\right]$. It is easily seen that both expressions are symmetric around $p=1/2$ and possess power series expansions around that point. More precisely, one finds that
\begin{align*}
\frac{1}{p(1-p)} =& 4\sum_{i=0}^\infty{2^{2i}\left(p-1/2\right)^{2i}},\\
\intertext{and}
n\left[\psi^{(1)}(1+np)+\psi^{(1)}(1+n(1-p))\right] =& 2\sum_{i=0}^\infty{\frac{n^{2i+1}}{(2i)!}\psi^{(2i+1)}(n/2+1)\left(p-1/2\right)^{2i}}.
\end{align*}
We can thus prove the claim by arguing that, for every positive integer $n$ and every non-negative integer $i$,
\begin{equation*}
\frac{n^{2i+1}}{(2i)!}\psi^{(2i+1)}(n/2+1)<2^{2i+1}.
\end{equation*}
It follows from \citet[Theorem 1]{chen2005inequalities} (applied with $m=0$) that
\begin{equation*}
\psi^{(2i+1)}(n/2+1) < \frac{(2i)!}{(n/2+1)^{2i+1}}\left[1+\frac{2i+1}{n+2}+\frac{(2i+1)(2i+2)}{3(n+2)^2}\right],
\end{equation*}
and it thus suffices to prove that
\begin{equation*}
(1+2/n)^{2i+1} > 1+\frac{2i+1}{n+2}+\frac{(2i+1)(2i+2)}{3(n+2)^2}.
\end{equation*}
This is obvious for $i=0$. For positive $i$, it follows from the observation that the left-hand side is no less than the sum of the first three terms of its binomial expansion, and that this sum exceeds the right-hand side.
\end{proof}

In the next lemma we establish conditions for a function $f:[0,1]\to[0,1]$ to have a unique fixed point, and for the fixed-point iteration to converge to this fixed-point from any starting value. Notably, we do not assume that $f$ is a contraction. This fixed-point theorem will be the key ingredient in our proof of \cref{thm-n-fixed}.

\begin{lemma}
\label{lem-fixedpoint}
Let $f:[0,1]\to[0,1]$ be a differentiable convex function such that $f(0)=f(1)=1$ and $\lim_{x\to 1}f'(x)>1$. Then $f$ has exactly one fixed point $x^*$ in $(0,1)$. If, moreover, $|f'(x^*)|<1$, $f$ is symmetric around $1/2$ and $c\coloneqq f(1/2)$ is such that either $c<1/2$ and $f(c)<1/2$, or $c\geq 1/2,$ then, for every $x\in(0,1)$, the sequence of function iterates $(x,f(x),f(f(x)),\ldots)$ converges to $x^*$.
\end{lemma}
\begin{proof}
Since $f(1)=1$ and $\lim_{x\to 1}f'(x)>1$, there exists $y\in(0,1)$, such that $f(y)<y$. If we define $g:x\mapsto f(x)-x$ on $[0,1]$, then $g(0)=1>0$ and $g(y)<0$, and thus, by the intermediate value theorem, there exists $x^*\in(0,y)$, such that $g(x^*)=0$, i.\,e. $x^*$ is a fixed point $f$. Uniqueness follows from the observation that $g(1)=0$ and the fact that the convex function $g$ can have at most two zeros.

For the second part of the lemma we observe that the condition $|f'(x^*)|<1$ implies via Banach's contraction principle \citep[Theorem 1.1]{granas2003fixed} that there exists a neighbourhood $I$ of $x^*$ such that the iteration of $f$ on $I$ converges to $x^*$. We can thus define $J=(a,b)$ as the maximal interval containing $I$ with the property that the iteration of $f$ on $J$ converges to $x^*$. We will show that $a=0$, $b=1$. It is clear from the continuity of $f$ that $J$ is open. By definition, the image of $J$ under $f$ is contained in $J$, and since $J$ is  maximal the images of the boundary points $a$, $b$ are not elements of $J$. This implies that $f(a), f(b)\in \{a,b\}$. Since $f(a)=a$ would imply the contradiction that $a\in\{x^*,1\}$, we have $f(a)=b$. We therefore need to rule out the possibility $f(b)=a$ so that it will follow that $f(b)=b$ and thus $b=1$ and $a=0$. 
We first consider the case $c=f(1/2)<1/2$, which implies $x^*<1/2$ and $1/2\in J$. To see this, we observe that the sequence $f^{2i}(1/2)$ of even function iterates is a decreasing sequence and bounded from below, and therefore converges to, say, $y$, which is a fixed point of $f\circ f$. The condition $f(c)<1/2$ implies that $f\circ f$ has only one fixed point in $(0,1)$ and since $x^*$ is a fixed point of $f\circ f$, it follows that $y=p$ and $1/2\in J$. Hence, by the symmetry of $f$, the maximal interval $J=(a,b)$ is symmetric around $1/2$ and one has $f(a)=f(b)$. This implies that $f(b)$ cannot equal $a$ because otherwise $b=f(a)=f(b)=a$. The argument is similar for the case $c\geq 1/2$ except that no assumption on $f(c)$ is necessary to conclude that $1/2\in J$.
\end{proof}
An illustration of the function $P_n$ and the fixed-point iteration considered in \cref{lem-fixedpoint} is provided in \cref{fig-Pnfixedpoint}. We now proceed to the proof of our second main result.
\begin{figure}
\centering
\begin{psfrags}
\input{./pnfixedpoint-psfrag.txt}
\includegraphics[width=0.75\textwidth]{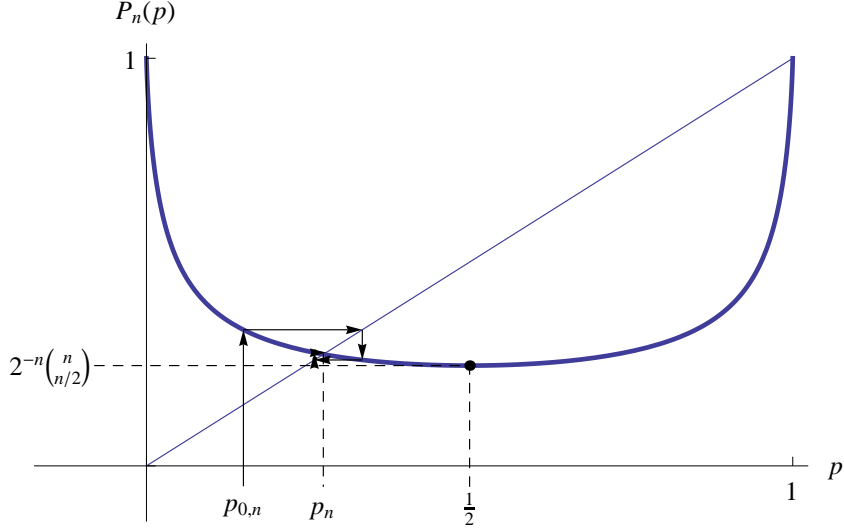}
\end{psfrags}
\caption{Illustration of the function $P_n$ governing the recursion \labelcref{eq-definition-pkn} and the fixed-point iteration considered in the proof of \cref{lem-fixedpoint}. In the picture, $n=10$ and $p_{0,n}=0.15$.}
\label{fig-Pnfixedpoint}
\end{figure}

\begin{proof}[Proof of \cref{thm-n-fixed}]
We first observe that the functions $P_n$ satisfy
\begin{equation*}
\lim_{p\to0}P_n(p) = \lim_{p\to1}P_n(p) = 1, \quad \lim_{p\to0}P_n'(p) = -\infty, \quad \lim_{p\to1}P_n'(p) = +\infty,
\end{equation*}
and apply \cref{lem-convexity,lem-fixedpoint} to conclude that they possess a unique fixed point $p_n\in(0,1)$. In order to establish the convergence $p_{k,n}\to p_n$, we need to verify the assumptions of the second part of \cref{lem-fixedpoint}.

We first need to analyse in more detail the point $(c_n,P_n(c_n))$, where $c_n=P_n(1/2)=2^{-n}\binom{n}{n/2}$. For $n=1,2$ one obtains $c_1=2/\pi>1/2$ and $c_2=1/2$, respectively, and the conclusion follows. For $n\geq 2$, the numbers $c_n$ are less than $1/2$ and we thus need to show that $P_n(c_n)<1/2$. This is easily checked numerically for $n=3$. For $n>3$, it follows from an application of the estimate
\begin{equation*}
\frac{2^n}{\sqrt{\pi(n+1)/2}} \leq \binom{n}{n/2} \leq \frac{2^n}{\sqrt{\pi n/2}}
\end{equation*}
to the central binomial coefficients followed by an application of Stirling's bounds (\cref{lem-Stirling-approx}) to $\binom{n}{\sqrt{2n/\pi}}$ that
\begin{align*}
P_n(c_n) =& \binom{n}{n2^{-n}\binom{n}{n/2}}\left[2^{-n} \binom{n}{n/2}\right]^{n2^{-n}\binom{n}{n/2}} \left[1-2^{-n} \binom{n}{n/2}\right]^{n\left(1-2^{-n}\binom{n}{n/2}\right)}\\
\leq & \binom{n}{\sqrt{2 n/\pi}} \left(\frac{2}{\pi  n}\right)^{\sqrt{\frac{n}{2 \pi }}} \left(1-\sqrt{\frac{2}{\pi  (n+1)}}\right)^{n \left(1-\sqrt{\frac{2}{\pi  (n+1)}}\right)}\\
\leq & \frac{\ee}{\sqrt[4]{2^5 \pi ^3n}}\left(1-\sqrt{\frac{2}{\pi  n}}\right)^{\sqrt{\frac{2 n}{\pi }}-n-\frac{1}{2}} \left(1-\sqrt{\frac{2}{\pi  (n+1)}}\right)^{n \left(1-\sqrt{\frac{2}{\pi  (n+1)}}\right)}.
\end{align*}
To get from the first to the second line, we used that both $\binom{n}{n/2}$ and its bound $2^n/\sqrt{\pi n/2}$ are less than $2^{n-1}$ and that $x\mapsto\binom{n}{2^{-n} n x}$ is increasing for $x<2^{n-1}$. Showing that the last line in the previous display is less than $1/2$ for all $n$ greater than three is a matter of basic, yet tedious, calculations. We next prove that $|P_n'(p_n)|<1$. Differentiation of $P_n$ and simplification of the resulting expression using $p_n=P_n(p_n)$ shows that
\begin{equation}
\label{eq-Pnprimepn}
P_n'(p_n)=n p_n \left(\psi^{(0)}(1+n(1-p_n))-\psi^{(0)}(1+np_n)-\log[(1-p)/p]\right).
\end{equation}
We will only consider the case $n>2$. For $n=1,2$ the claim can be checked numerically or dealt with by a straightforward adaptation of the arguments we are about to present. The proof so far has shown that, for $n>2$, the fixed point $p_n$ is less that $1/2$, and that $P_n'(p_n)$ is thus negative. To show that the right-hand side of \cref{eq-Pnprimepn} exceeds $-1$, we use the bounds \citep[Theorem 1]{chen2005inequalities}
\begin{equation*}
\log x-\frac{1}{2x} - \frac{1}{6x^2}\leq \psi^{(0)}(x) \leq \log x-\frac{1}{2x},
\end{equation*}
and 
\begin{equation*}
1-\frac{1}{1-x}\leq \log (1+x) \leq x,
\end{equation*}
which are valid for positive $x$. We thus obtain
\begin{equation*}
P_n'(p_n)\geq -1+\frac{1}{2}\frac{1+n}{1+n(1-p_n)}+\frac{1}{6}\frac{1}{(1+n p_n)^2}-\frac{2}{3}\frac{1}{1+n p_n}\eqqcolon -1+r_n(p_n).
\end{equation*}
It is an easy exercise to show that $r_n(p)$ is greater than zero for all $p\in(0,1)$.

The final statement of the theorem about the decay rate of $p_n$ as $n$ tends to infinity follows again from applying \cref{eq-Stirling-approx} to \cref{eq-fixed-point}, which yields
\begin{equation*}
p_n \sim \frac{1}{\sqrt{2\pi p_n(1-p_n)}\sqrt{n}},
\end{equation*}
and thus $p_n\sim (2\pi n)^{-1/3}$.
\end{proof}


\end{document}